\documentclass[12pt]{amsart}
\usepackage{amsmath, amsthm, amsfonts, amscd, amssymb, eucal, latexsym, mathrsfs, appendix, parskip, tikz}
\usepackage[numbers,sort&compress]{natbib}
\usepackage{color,hyperref}
\hypersetup{colorlinks,breaklinks,
            linkcolor=red,urlcolor=red,
            anchorcolor=red,citecolor=blue}

\theoremstyle{plain}
\newtheorem*{rep@theorem}{\rep@title}
\newcommand{\newreptheorem}[2]{%
\newenvironment{rep#1}[1]{%
 \def\rep@title{#2 \ref{##1}}%
 \begin{rep@theorem}}%
 {\end{rep@theorem}}}
\makeatother

\newreptheorem{theorem}{Theorem}

\newtheorem{theorem}{Theorem}[section]

\newtheorem{lemma}[theorem]{Lemma}
\newtheorem{proposition}[theorem]{Proposition}

\newtheorem{definition}[theorem]{Definition}
\newtheorem{remark}[theorem]{Remark}

\newtheorem*{conjecture*}{Conjecture}

\allowdisplaybreaks

\begin{document}
\title{Fourier coefficients of $\times p$-invariant measures}
\author{Huichi Huang}
\address{Huichi Huang, College of Mathematics and Statistics, Chongqing University, Chongqing, 401331, PR China}
\email{huanghuichi@cqu.edu.cn}
\date{\today}
\begin{abstract}
We consider densities $D_\Sigma(A)$, $\overline{D}_\Sigma(A)$ and $\underline{D}_\Sigma(A)$ for a subset $A$ of $\mathbb{N}$ with respect to a  sequence $\Sigma$ of finite subsets of $\mathbb{N}$ and study Fourier coefficients of ergodic, weakly mixing and strongly mixing $\times p$-invariant measures on the unit circle $\mathbb{T}$. Combining these, we  prove the following measure rigidity results: on $\mathbb{T}$, the Lebesgue measure is the only non-atomic  $\times p$-invariant measure   satisfying one of the following: (1) $\mu$ is ergodic and there exist a F\o lner sequence $\Sigma$ in $\mathbb{N}$ and a nonzero integer $l$ such that $\mu$ is $\times (p^j+l)$-invariant for  all $j$ in a subset $A$ of $\mathbb{N}$ with $D_\Sigma(A)=1$; (2) $\mu$ is weakly mixing and there exist a F\o lner sequence $\Sigma$ in $\mathbb{N}$ and a nonzero integer $l$ such that $\mu$ is $\times (p^j+l)$-invariant for  all $j$ in a subset $A$ of $\mathbb{N}$ with $\overline{D}_\Sigma(A)>0$; (3) $\mu$ is strongly mixing and there exists a nonzero integer $l$ such that $\mu$ is  $\times (p^j+l)$-invariant for  infinitely many $j$. Moreover, a $\times p$-invariant measure satisfying (2) or (3) is either a Dirac measure or the Lebesgue measure.

As an application  we prove that  for every increasing function $\tau$ defined on positive integers with $\lim_{n\to\infty}\tau(n)=\infty$, there exists a multiplicative semigroup $S_\tau$ of $\mathbb{Z}^+$ containing $p$ such that $|S_\tau\cap[1,n]|\leq (\log_p n)^{\tau(n)}$ and the Lebesgue measure is the only non-atomic ergodic $\times p$-invariant measure which is $\times q$-invariant for all $q$ in $S_\tau$.
\end{abstract}

\maketitle
\section{Introduction}

There are two motivations for this paper. Both  are related to the celebrated $\times p,\times q$ conjecture by H. Furstenberg. The first motivation is Lyons' Theorem and Rudolph-Johnson's Theorem, and the second is a theorem due to E. A. Sataev and later independently discovered by M. Einsiedler and A. Fish.

For an  integer $p$, consider the group homomorphism $T_p$~(called the $\times p$ map) on the unit circle $\mathbb{T}=\mathbb{R}/ \mathbb{Z}$ given by $T_p(x)=px \mod\mathbb{Z}$ for  all $x$ in $\mathbb{R}/ \mathbb{Z}$.

When $p$ and $q$ are positive integers greater than 1 with $\frac{\log p}{\log q}\notin\mathbb{Q}$,   H. Furstenberg gave a classification of  $\times p,\times q$-invariant closed subsets in $\mathbb{T}$~\cite[Thm. IV.1]{Furstenberg1967}.
\begin{theorem}~[Furstenberg, 1967]\

A $\times p,\times q$-invariant closed subset in $\mathbb{T}$ is either finite or $\mathbb{T}$.
\end{theorem}
Motivated by this, H. Furstenberg  conjectured a classification of  $\times p,\times q$-invariant measures.

\begin{conjecture*}~[Furstenberg's $\times p,\times q$ conjecture]\
~\label{conj: F}

For two positive integers $p,q\geq 2$ with $\frac{\log p}{\log q}\notin\mathbb{Q}$, an ergodic $\times p,\times q$-invariant measure on $\mathbb{T}$ is either finitely supported or the Lebesgue measure. That is, the only non-atomic $\times p,\times q$-invariant measure on $\mathbb{T}$ is the Lebesgue measure.
\end{conjecture*}

The first progress was made by R. Lyons in 1988~\cite[Thm. 1]{Lyons1988}.
\begin{theorem}~[Lyons' theorem]\

Suppose $p,q$ are relatively prime. The Lebesgue measure is the only  non-atomic $\times p,\times q$-invariant measure  which is $T_p$-exact.
\end{theorem}
A measure $\mu$ is $T_p$-exact means $h_\mu(T_p,\xi)>0$ for any nontrivial finite partition $\xi$ of $\mathbb{T}$, where $h_\mu(T_p,\xi)$ stands for the measure entropy of $T_p$ with respect to a finite partition $\xi$.

In 1990, D. Rudolph improved Lyons' theorem to the following~\cite[Thm. 4.9]{Rudolph1990}.
\begin{theorem}~[Rudolph's theorem]\

Suppose $p,q$ are relatively prime.  A  $\times p, \times q$-invariant measure $\mu$ with $h_\mu(T_p)=\sup_\xi h_\mu(T_p,\xi)>0$ must be the Lebesgue measure.
\end{theorem}

Rudolph's theorem  was strengthened by  A. S. A. Johnson~\cite[Thm. A]{Johnson1992}.
\begin{theorem}~[Rudolph-Johnson's theorem]\

Suppose that $\frac{\log p}{\log q}\notin\mathbb{Q}$.  Then a  $\times p, \times q$-invariant measure with $\displaystyle h_\mu(T_p)>0$ is the Lebesgue measure.
\end{theorem}

 In this paper  by assuming that a non-atomic $\times p$-invariant measure $\mu$  satisfies  weaker conditions than $T_p$-exactness or positive entropy, we prove that if
 $\mu$ is invariant under enough many $\times q$-maps of special forms, then $\mu$ is the Lebesgue measure.

\begin{reptheorem}{thm: mr}
The Lebesgue measure is the only non-atomic $\times p$-invariant measure on $\mathbb{T}$ satisfying one of the following:
\begin{enumerate}
\item it is ergodic and there exist a nonzero integer $l$ and a F\o lner sequence $\Sigma=\{F_n\}_{n=1}^\infty$ in $\mathbb{N}$ such that
$\mu$ is $\times (p^j+l)$-invariant for all $j$ in some $A\subseteq\mathbb{N}$ with $D_\Sigma(A)=1$;
\item it is weakly mixing and there exist a nonzero integer $l$ and a F\o lner sequence $\Sigma=\{F_n\}_{n=1}^\infty$ in $\mathbb{N}$ such that
$\mu$ is $\times (p^j+l)$-invariant for all $j$ in some $A\subseteq\mathbb{N}$ with $\overline{D}_\Sigma(A)>0$;
\item it is strongly mixing and there exist a nonzero integer $l$ and an infinite set $A\subseteq \mathbb{N}$ such that $\mu$ is $\times  (p^j+l)$-invariant for all $j$ in $A$.
\end{enumerate}
Moreover, a $\times p$-invariant measure satisfying (2) or (3) is either a Dirac measure  or the Lebesgue measure.
\end{reptheorem}
Here $D_\Sigma(A)$ and $\overline{D}_\Sigma(A)$ are density and upper density of $A$ with respect to $\Sigma$ respectively. See  Section 2 for their definitions.

The second motivation is a theorem which is independently discovered by E. A. Sataev in 1975~\cite[Thm. 1]{Sataev1975} and M. Einsiedler and A. Fish in 2010~\cite[Thm. 1.2]{EisiendlerFish2010}.
\begin{theorem}

~\label{thm: EF}
For a multiplicative semigroup $S$ of positive integers with
$$\liminf_{n\to\infty}\frac{\log|S\cap[1,n]|}{\log n}>0,$$ if a Borel probability measure on $\mathbb{T}$ is an ergodic $\times p$-invariant measure for some $p$ in $S$ and  is $\times q$-invariant for every $q$ in $S$, then it  is either finitely supported or Lebesgue measure.
\end{theorem}

As an application of Theorem~\ref{thm: mr}, we prove that  there exists  a multiplicative semigroup $S$ of positive integers with $\displaystyle\lim_{n\to\infty}\frac{\log|S\cap[1,n]|}{\log n}=0$ such that Theorem~\ref{thm: EF} still holds~(see Theorem~\ref{prop: sg}).

The paper is organized as follows.

Firstly we give definitions of density functions of a subset $A$ of nonnegative integers with respect to a F\o lner sequence. In Section 3, we lay down some basic facts about Fourier coefficients of a measure on the unit circle. In Section 4, we give the characterizations of ergodic, weakly mixing and strongly mixing $\times p$-invariant measures via their Fourier coefficients. In the last section, we prove the main theorem, Theorem~\ref{thm: mr}. Applying it, we prove Theorem~\ref{prop: sg} at the end.

\section*{Acknowledgements}
This paper was started when I was a postdoctoral fellow under ERC Advanced Grant No. 267079 in University of M\"{u}nster.  I am grateful to Joachim Cuntz for his support and addressing me the problem of $\times p,\times q$.

Part of work was carried out during a visit to Shanghai in Spring 2016.
I thank all members in Research Center for Operator Algebras in East China Normal University,  in particular,  Huaxin Lin for warm hospitality. I  thank Kunyu Guo and Kai Wang for their support during my stay in Fudan University.

Thanks to suggestions of  the anonymous referee and Anatole Katok, various changes are made. In particular, they kindly  helps me to improve Theorem~\ref{prop: sg}.

\section{Preliminaries}

Let $\mathbb{N}$ stand for the set of nonnegative integers and $\mathbb{Z}^+$ stand for the set of positive integers. Throughout this article, for two integers $a<b$, we denote the set $\{a,\cdots, b\}$ by $[a, b]$. Denote by $|F|$ the cardinality of  a set $F$.

The following definition of F\o lner sequence in $\mathbb{N}$ is a special case of F\o lner sequences in an amenable semigroup~\cite[p.2]{Bowley1971}.

\begin{definition}
A sequence $\Sigma=\{F_n\}_{n=1}^\infty$ of finite subsets in $\mathbb{N}$ is called a {\bf F\o lner sequence}
if $$\displaystyle\lim_{n\to\infty}\dfrac{|(F_n+m)\Delta F_n|}{|F_n|}=0$$ for every $m$ in $\mathbb{N}$.
\end{definition}

The density $D(A)$ of a subset $A$ of $\mathbb{N}$ is given by $\displaystyle D(A)=\lim_{n\to\infty}\frac{|A\cap [0,n-1]|}{n}$. The upper density of $A$, $\displaystyle\overline{D}(A):=\limsup_{n\to\infty}\frac{|A\cap [0,n-1]|}{n}$ and the lower density of $A$, $\displaystyle\underline{D}(A):=\liminf_{n\to\infty}\frac{|A\cap [0,n-1]|}{n}$.  These densities are defined via the sequence  of finite subsets $\{[0, n-1]\}_{n=1}^\infty$ in $\mathbb{N}$.
Generalizing these, one can define  densities of $A$ with respect to every  sequence of finite subsets of $\mathbb{N}$.

\begin{definition}
Let $\Sigma=\{F_n\}_{n=1}^\infty$ be a sequence of finite subsets of $\mathbb{N}$. The density $D_\Sigma(A)$ of a subset $A$ of $\mathbb{N}$ with respect to $\Sigma$ is given by
$$\displaystyle D_\Sigma(A)=\lim_{n\to\infty}\frac{|A\cap F_n|}{|F_n|}.$$
The upper density $\displaystyle \overline{D}_\Sigma(A):=\limsup_{n\to\infty}\frac{|A\cap F_n|}{|F_n|}$, and the lower density $\displaystyle \underline{D}_\Sigma(A):=\liminf_{n\to\infty}\frac{|A\cap F_n|}{|F_n|}$.
\end{definition}

\begin{remark}
\begin{enumerate}
\item Denote $\cup F_n$ by $F$. Then $D_\Sigma(A)=D_\Sigma(A\cap F)$.
\item The density $D_\Sigma(A)$ depends on choices of $\Sigma$.    For instance, let $\displaystyle A=\cup_{n=1}^\infty [2^n, 2^n+n]$. For the F\o lner sequence
$\Sigma=\{F_m\}_{m=1}^\infty$ with $F_m=[1,m]$,  one has $D_\Sigma(A)=0$. On the other hand $D_{\Sigma'}(A)=1$ for the F\o lner sequence $\Sigma'=\{[2^n, 2^n+n]\}_{n=1}^\infty$.
\end{enumerate}
\end{remark}

Within this paper,  a measure on a compact metrizable $X$ always means a Borel probability measure.  A measure $\mu$  is called {\bf non-atomic} if $\mu\{x\}=0$ for every $x$ in $X$.

A  topological dynamical system consists of a compact metrizable space $X$ and a continuous map $T:X\to X$.

A measure $\mu$ on $X$ is called {\bf $T$-invariant} if $\mu(B)=\mu(T^{-1} B)$ for any Borel subset $B$ of $X$. A $T$-invariant measure $\mu$ is called {\bf ergodic} if every Borel subset B with $T^{-1}B=B$ satisfies that $\mu(B)^2=\mu(B)$,    it is called {\bf weakly mixing} if $\mu\times\mu$ is an ergodic $T\times T$-invariant measure on $X\times X$, and  it is called {\bf strongly mixing} if $\displaystyle \lim_{j\to\infty} \mu(T^{-j}A\cap B)=\mu(A)\mu(B)$ for all Borel subsets $A, B$ in $X$.

It's well-known that strongly mixing$\Longrightarrow$weakly mixing$\Longrightarrow$ergodic.

Within this paper, we only consider that $X=\mathbb{T}$ and $T=T_p$ is the $\times p$ map on $\mathbb{T}=\mathbb{R}/ \mathbb{Z}$ defined by $T_p(x)=px\mod\mathbb{Z}$ for all $x$ in $\mathbb{R}/ \mathbb{Z}$ and $p$ in $\mathbb{Z}$.

\section{Some basic facts about Fourier coefficients}\

Denote the support of $\mu$ by ${\rm Supp}(\mu)$.  For $n$ in $\mathbb{Z}$, the {\bf Fourier coefficient} $\hat{\mu}(n)$ of a measure $\mu$ on $\mathbb{T}$ is given by $\hat{\mu}(n)=\int_\mathbb{T} z^n\,d\mu(z)$ when taking
$\mathbb{T}=\{z\in\mathbb{C}| |z|=1\}$.
\begin{lemma}~\label{f0}\
For nonzero  $k$ in $\mathbb{Z}$ and $c$ in $\mathbb{T}$, $\int_{\mathbb{T}}|z^k-c|^2\,d\mu(z)=0$  if and only if  ${\rm Supp}(\mu)\subseteq \{z|z^k=c\}$.
\end{lemma}
\begin{proof}
Obvious.
\end{proof}

\begin{proposition}~\label{Fouriercoefficients}\
 For a nonzero integer $k$, one has $|\hat{\mu}(k)|<1$  if and only if there is no $c$ in $\mathbb{T}$ such that ${\rm Supp}(\mu)\subseteq \{z|z^k=c\}$.

\end{proposition}
\begin{proof}
Let $k$ be a nonzero integer. By Lemma~\ref{f0}, it suffices to show that  $|\hat{\mu}(k)|<1$  if and only if $\int_{\mathbb{T}}|z^k-c|^2\,d\mu(z)>0$ for all $c\in\mathbb{T}$.

If  $|\hat{\mu}(k)|<1$, then for any $c\in\mathbb{T}$, we have
\begin{align*}
&\int_{\mathbb{T}}|z^k-c|^2\,d\mu(z)=\int_{\mathbb{T}}(z^k-c)(\bar{z}^k-\bar{c})d\mu(z)  \\
&=2-2Re(\bar{c}\hat{\mu}(k))\geqslant 2-2|\hat{\mu}(k)|>0.
\end{align*}

Conversely assume that $\int_{\mathbb{T}}|z^k-c|^2\,d\mu(z)>0$ for all $c\in\mathbb{T}$. When choosing $c\in\mathbb{T}$ such that $c\hat{\mu}(k)=|\hat{\mu}(k)|$, we get
$0<\int_{\mathbb{T}}|z^k-c|^2\,d\mu(z)=2-2|\hat{\mu}(k)|$, which implies that $|\hat{\mu}(k)|<1$.
\end{proof}

\section{Fourier coefficients of ergodic, weakly mixing or strongly mixing $\times p$-invariant measures}

In this section, we give  characterizations of ergodic, weakly mixing and strongly mixing $\times p$-invariant measures via their Fourier coefficients.
\begin{theorem}~\label{thm: F}
The following are true.
\begin{enumerate}
\item A measure $\mu$  on $\mathbb{T}$ is an ergodic $\times p$-invariant measure  if and only if
$$\displaystyle \lim_{n\to\infty}\frac{1}{|F_n|}\sum_{j\in F_n}\hat{\mu}(kp^j+l)=\hat{\mu}(k)\hat{\mu}(l)$$ for every F\o lner sequence $\Sigma=\{ F_n\}_{n=1}^\infty$ in
$\mathbb{N}$ and all $k,l$ in $\mathbb{Z}$.
\item  A measure $\mu$ on $\mathbb{T}$ is a weakly mixing $\times p$-invariant measure  if and only if
$$\displaystyle \lim_{n\to\infty}\frac{1}{|F_n|}\sum_{j\in F_n}|\hat{\mu}(kp^j+l)-\hat{\mu}(k)\hat{\mu}(l)|^2=0$$ for every F\o lner sequence $\Sigma=\{ F_n\}_{n=1}^\infty$ in
$\mathbb{N}$ and all $k,l$ in $\mathbb{Z}$.
\item  A measure $\mu$ on $\mathbb{T}$ is a strongly mixing $\times p$-invariant measure  if and only if
$$\displaystyle\lim_{j\to\infty}\hat{\mu}(kp^j+l)=\hat{\mu}(k)\hat{\mu}(l)$$ for all $k,l$ in $\mathbb{Z}$.
\end{enumerate}
\end{theorem}

To prove (1) and (2) of Proposition~\ref{thm: F}, we need a preliminary result, which is a special case of von  Neumann's mean ergodic theorem for amenable semigroups proved by Bowley~\cite[Thm. 1]{Bowley1971}.
\begin{lemma}~\label{lem: met}
For a topological dynamical system $(X, T)$, if $\nu$ is an ergodic $T$-invariant measure on $X$, then for every F\o lner sequence $\{F_n\}_{n=1}^\infty$ in $\mathbb{N}$, one has
$$\displaystyle \lim_{n\to\infty} \frac{1}{|F_n|} \sum_{j\in F_n} f(T^jx)=\int_X f\, d\nu$$ for every $f$ in $L^2(X,\nu)$~(note that the identity holds with respect to $L^2$-norm).
Consequently $$\displaystyle \lim_{n\to\infty} \frac{1}{|F_n|} \sum_{j\in F_n}\int_X f(T^jx)g(x)\,d\nu(x)=\int_X f\, d\nu\,\int_X g\,d\nu$$ for every $f,g$ in $L^2(X,\nu)$
\end{lemma}

\begin{proof}~[Proof of Theorem~\ref{thm: F}]
\begin{enumerate}
\item Suppose $\mu$ is an ergodic $\times p$-invariant measure on $\mathbb{T}$. Denote the $\times p$ map by $T_p$. Consider the measurable dynamical system
$(\mathbb{T}, T_p,\mu)$.  Using Lemma~\ref{lem: met}, we get
$$\lim_{n\to\infty} \frac{1}{|F_n|}\sum_{j\in F_n}\int_\mathbb{T} f(T_p^j(x))g(x)\,d\mu(x)=\int_\mathbb{T} f\, d\mu\int_\mathbb{T} g\, d\mu$$ for all continuous functions $f,g$ on $\mathbb{T}$.  By choosing $f=z^k$ and $g=z^l$, we prove the necessity.

Now assume that $\displaystyle \lim_{n\to\infty}\frac{1}{|F_n|}\sum_{j\in F_n}\hat{\mu}(kp^j+l)=\hat{\mu}(k)\hat{\mu}(l)$ for every F\o lner sequence $\{ F_n\}_{n=1}^\infty$ in $\mathbb{N}$ and all $k,l$ in $\mathbb{Z}$. Let $l=0$, we have $$\displaystyle \lim_{n\to\infty}\frac{1}{|F_n|}\sum_{j\in F_n}\hat{\mu}(kp^j)=\hat{\mu}(k)$$ for every $k$ in
$\mathbb{Z}$. Replacing $k$ by $kp$, one has
\begin{align*}
\hat{\mu}(kp)&=\lim_{n\to\infty}\frac{1}{|F_n|}\sum_{j\in F_n}\hat{\mu}(kp^{j+1})  \\
&=\lim_{n\to\infty}\frac{1}{|F_n+1|}\sum_{j\in F_n+1}\hat{\mu}(kp^j)  \\
&\tag{$\{F_n+1\}_{n=1}^\infty$ is a F\o lner sequence in $\mathbb{N}$.} \\
&=\hat{\mu}(k)
\end{align*}
for every $k$ in $\mathbb{N}$.
Hence $\mu$ is $\times p$-invariant.

From $\displaystyle \lim_{n\to\infty}\frac{1}{|F_n|}\sum_{j\in F_n}\hat{\mu}(kp^j+l)=\hat{\mu}(k)\hat{\mu}(l)$, we have
$\displaystyle \lim_{n\to\infty}\frac{1}{|F_n|}\sum_{j\in F_n}\int_\mathbb{T} f(T_p^jx)g(x)\,d\mu(x)=\int_\mathbb{T}f\,\int_\mathbb{T} g$ for all polynomials on $\mathbb{T}$. Polynomials are dense in  $L^2(\mathbb{T},\mu)$,  so  $\displaystyle \lim_{n\to\infty}\frac{1}{|F_n|}\sum_{j\in F_n}\int_\mathbb{T} f(T_p^jx)g(x)\,d\mu(x)=\int_\mathbb{T}f\,d\mu\int_\mathbb{T} g\,d\mu$ for all $f,g\in L^2(\mathbb{T},\mu)$. In particular, it is true for $f=g=1_A$ for a Borel subset $A$ with $T_p^{-1}A=A$. Hence $\mu(A)=\mu(A)^2$. This proves that $\mu$ is ergodic.

\item Suppose $\mu$ is a weakly mixing $\times p$-invariant measure on $\mathbb{T}$, which means, $\mu\times\mu$ is an ergodic $T_p\times T_p$-invariant measure on $\mathbb{T}^2$. Applying the second identity of Lemma~\ref{lem: met} to $X=\mathbb{T}^2$, $\nu=\mu\times\mu$ and letting $f(z_1,z_2)=z_1^kz_2^{-k}$ and $g(z_1,z_2)=z_1^l z_2^{-l}$ for any $k,l$ in $\mathbb{Z}$, we have
$$\displaystyle \lim_{n\to\infty}\frac{1}{|F_n|}\sum_{j\in F_n}|\hat{\mu}(kp^j+l)|^2=|\hat{\mu}(k)|^2 |\hat{\mu}(l)|^2.$$
Note that
\begin{align*}
&|\hat{\mu}(kp^j+l)-\hat{\mu}(k)\hat{\mu}(l)|^2  \\
=&|\hat{\mu}(kp^j+l)|^2+|\hat{\mu}(k)|^2 |\hat{\mu}(l)|^2-\hat{\mu}(kp^j+l)\hat{\mu}(-k)\hat{\mu}(-l)-\hat{\mu}(-kp^j-l)\hat{\mu}(k)\hat{\mu}(l).
\end{align*}

So we get
\begin{align*}
&\lim_{n\to\infty}\frac{1}{|F_n|}\sum_{j\in F_n}|\hat{\mu}(kp^j+l)-\hat{\mu}(k)\hat{\mu}(l)|^2 \\
=&\lim_{n\to\infty}\frac{1}{|F_n|}\sum_{j\in F_n}[|\hat{\mu}(kp^j+l)|^2+|\hat{\mu}(k)|^2 |\hat{\mu}(l)|^2-\hat{\mu}(kp^j+l)\hat{\mu}(-k)\hat{\mu}(-l)-\hat{\mu}(-kp^j-l)\hat{\mu}(k)\hat{\mu}(l)]  \\
\tag{Use (1) since  that $\mu$ is weakly mixing implies ergodicity of $\mu$.}  \\
=&|\hat{\mu}(k)|^2 |\hat{\mu}(l)|^2+|\hat{\mu}(k)|^2 |\hat{\mu}(l)|^2-|\hat{\mu}(k)|^2 |\hat{\mu}(l)|^2-|\hat{\mu}(k)|^2 |\hat{\mu}(l)|^2=0
\end{align*}
for all $k,l$ in $\mathbb{Z}$.

On the other hand, suppose $$\displaystyle \lim_{n\to\infty}\frac{1}{|F_n|}\sum_{j\in F_n}|\hat{\mu}(kp^j+l)-\hat{\mu}(k)\hat{\mu}(l)|^2=0$$ for every F\o lner sequence $\{ F_n\}_{n=1}^\infty$ in $\mathbb{N}$ and all $k,l$ in $\mathbb{Z}$.

Firstly we have $$\displaystyle \lim_{n\to\infty}\frac{1}{|F_n|}\sum_{j\in F_n}\hat{\mu}(kp^j+l)=\hat{\mu}(k)\hat{\mu}(l).$$ So by (1) $\mu$ is an ergodic $\times p$-invariant measure. To prove $\mu\times \mu$ is an ergodic $T_p\times T_p$-invariant measure on $\mathbb{T}^2$, it suffices to prove that
$$\lim_{n\to\infty}\frac{1}{|F_n|}\sum_{j\in F_n}\int_{\mathbb{T}^2} \,f((T_p\times T_p)^j(z_1,z_2))g(z_1,z_2)\,d\mu(z_1)\,d\mu(z_2)=\int_{\mathbb{T}^2} f\,d\mu\,d\mu \int_{\mathbb{T}^2} g\,d\mu\,d\mu$$ for all continuous functions $f$ and $g$ on $\mathbb{T}^2$, which is equivalent to that
$$\lim_{n\to\infty}\frac{1}{|F_n|}\sum_{j\in F_n}\hat{\mu}(k_1p^j+l_1)\hat{\mu}(k_2p^j+l_2)=\hat{\mu}(k_1)\hat{\mu}(k_2)\hat{\mu}(l_1)\hat{\mu}(l_2)$$ for all $k_1,k_2,l_1,l_2$ in $\mathbb{Z}$ by letting $f=z_1^{k_1}z_2^{k_2}$ and $g=z_1^{l_1}z_2^{l_2}$ whose linear spans are dense in $C(\mathbb{T}^2)$.

Note that
\begin{align*}
&|\hat{\mu}(k_1p^j+l_1)\hat{\mu}(k_2p^j+l_2)-\hat{\mu}(k_1)\hat{\mu}(k_2)\hat{\mu}(l_1)\hat{\mu}(l_2)| \\
\leq &|\hat{\mu}(k_1p^j+l_1)[\hat{\mu}(k_2p^j+l_2)-\hat{\mu}(k_2)\hat{\mu}(l_2)]|+|[\hat{\mu}(k_1p^j+l_1)-\hat{\mu}(k_1)\hat{\mu}(l_1)]\hat{\mu}(k_2)\hat{\mu}(l_2)|  \\
\leq &|\hat{\mu}(k_2p^j+l_2)-\hat{\mu}(k_2)\hat{\mu}(l_2)|+|\hat{\mu}(k_1p^j+l_1)-\hat{\mu}(k_1)\hat{\mu}(l_1)|
\end{align*}
for all $k_1,k_2,l_1,l_2$ in $\mathbb{Z}$.

Hence we have
\begin{align*}
&\lim_{n\to\infty}\frac{1}{|F_n|}\sum_{j\in F_n}|\hat{\mu}(k_1p^j+l_1)\hat{\mu}(k_2p^j+l_2)-\hat{\mu}(k_1)\hat{\mu}(k_2)\hat{\mu}(l_1)\hat{\mu}(l_2)|^2  \\
\leq &\lim_{n\to\infty}\frac{1}{|F_n|}\sum_{j\in F_n} [|\hat{\mu}(k_2p^j+l_2)-\hat{\mu}(k_2)\hat{\mu}(l_2)|+|\hat{\mu}(k_1p^j+l_1)-\hat{\mu}(k_1)\hat{\mu}(l_1)|]^2 \\
\tag{Cauchy-Schwarz inequality}  \\
\leq & 2\lim_{n\to\infty}\frac{1}{|F_n|}\sum_{j\in F_n} [|\hat{\mu}(k_2p^j+l_2)-\hat{\mu}(k_2)\hat{\mu}(l_2)|^2+|\hat{\mu}(k_1p^j+l_1)-\hat{\mu}(k_1)\hat{\mu}(l_1)|^2]
=&0.
\end{align*}
Using the inequality
$(\frac{|x_1|+\cdots+|x_n|}{n})^2\leq \frac{|x_1|^2+\cdots+|x_n|^2}{n}$, we have
$$\lim_{n\to\infty}\frac{1}{|F_n|}\sum_{j\in F_n}|\hat{\mu}(k_1p^j+l_1)\hat{\mu}(k_2p^j+l_2)-\hat{\mu}(k_1)\hat{\mu}(k_2)\hat{\mu}(l_1)\hat{\mu}(l_2)|=0.$$
This completes the proof.

\item If $\mu$ is strongly mixing, then $\displaystyle\lim_{j\to\infty} \mu(T_p^{-j}A\cap B)=\mu(A)\mu(B)$ for all Borel subsets $A$ and $B$ in $\mathbb{T}$.  This means
$\displaystyle\lim_{j\to\infty} \int_\mathbb{T} 1_A(T_p^j x)1_B(x)\,d\mu(x)=\int_\mathbb{T} 1_A\,d\mu\int_\mathbb{T} 1_B\,d\mu$ for all Borel subsets $A$ and $B$, where $1_A$ stands for the characteristic function of $A$.

 Note that linear combinations of characteristic functions are dense in $L^2(\mathbb{T},\mu)$, so $\lim_{j\to\infty} f(T_p^j x)g(x)\,d\mu(x)=\int_\mathbb{T} f\,d\mu\int_\mathbb{T} g\,d\mu$ for all $f,g$ in $C(\mathbb{T})$. In particular, this holds for $f=z^k$ and $g=z^l$ for all $k,l$ in $\mathbb{Z}$, which means
$$\lim_{j\to\infty} \hat{\mu}(kp^j+l)=\hat{\mu}(k)\hat{\mu}(l)$$ for all $k,l\in\mathbb{Z}$.

 On the other hand, if a measure $\mu$ satisfies that
 $$\lim_{j\to\infty} \hat{\mu}(kp^j+l)=\hat{\mu}(k)\hat{\mu}(l)$$ for all $k,l\in\mathbb{Z}$. Let $l=0$ and replace $k$ by $kp$. Then we have
 $$\hat{\mu}(kp)=\lim_{j\to\infty} \hat{\mu}(kp^{j+1})=\hat{\mu}(k)$$ for all $k$ in $\mathbb{Z}$.

  Linear combinations of $z^k$ and $z^l$ are polynomials on $\mathbb{T}$, which is dense in $L^2(\mathbb{T},\mu)$. Hence
 $$\lim_{j\to\infty} \mu(f(T_p^j)g)=\mu(f)\mu(g)$$ for all $f,g$ in $L^2(\mathbb{T},\mu)$. In particular, it holds for $f=1_A$ and $g=1_B$ for any Borel subsets $A,B$ of
 $\mathbb{T}$, which completes the proof.

 \end{enumerate}
\end{proof}

\begin{remark}
\begin{enumerate}
\item As shown in~\cite{Lyons1988}, a measure $\mu$ is $T_p$-exact iff
$$\lim_{j\to\infty} \sup_{k\in\mathbb{Z}}|\hat{\mu}(kp^j+l)-\hat{\mu}(k)\hat{\mu}(l)|=0$$ for every $l$ in $\mathbb{Z}$. Hence $T_p$-exactness is much stronger than being strongly mixing.
\item So far it is unknown how to characterize that $h_\mu(T_p)>0$ via Fourier coefficients of $\mu$.
\end{enumerate}
\end{remark}

\section{Rigidity of $\times p$-invariant measures}

With the above preliminaries, we are ready to prove the main theorem.
\begin{theorem}~\label{thm: mr}
The Lebesgue measure is the only non-atomic $\times p$-invariant measure on $\mathbb{T}$ satisfying one of the following:
\begin{enumerate}
\item it is ergodic and there exist a nonzero integer $l$ and a F\o lner sequence $\Sigma=\{F_n\}_{n=1}^\infty$ in $\mathbb{N}$ such that
$\mu$ is $\times (p^j+l)$-invariant for all $j$ in some $A\subseteq\mathbb{N}$ with $D_\Sigma(A)=1$;
\item it is weakly mixing and there exist a nonzero integer $l$ and a F\o lner sequence $\Sigma=\{F_n\}_{n=1}^\infty$ in $\mathbb{N}$ such that
$\mu$ is $\times (p^j+l)$-invariant for all $j$ in some $A\subseteq\mathbb{N}$ with $\overline{D}_\Sigma(A)>0$;
\item it is strongly mixing and there exist a nonzero integer $l$ and an infinite set $A\subseteq \mathbb{N}$ such that $\mu$ is $\times  (p^j+l)$-invariant for all $j$ in $A$.
\end{enumerate}
Moreover, a $\times p$-invariant measures satisfying (2) or (3) is either a Dirac measure or the Lebesgue measure.
\end{theorem}

\begin{proof}~[Proof of the first part of Theorem~\ref{thm: mr}]

\begin{enumerate}
\item Suppose $\mu$ is an ergodic $\times p$-invariant measure  and  there exist a nonzero integer $l$ and a F\o lner sequence $\Sigma=\{F_n\}_{n=1}^\infty$ such that
$\mu$ is $\times (p^j+l)$-invariant for all $j$ in some $A\subseteq\mathbb{N}$ with $D_\Sigma(A)=1$.

If $\mu$ is not Lebesgue measure, then there exists nonzero $k$ in $\mathbb{Z}$ such that $0\hat{\mu}(k)$ is nonzero.

By Theorem~\ref{thm: F}(1), one has
$$\lim_{n\to\infty}\frac{1}{|F_n|}\sum_{j\in F_n}\hat{\mu}(kp^j+kl)=\hat{\mu}(k)\hat{\mu}(kl).$$
Note that
\begin{align*}
&\frac{1}{|F_n|}\sum_{j\in F_n}\hat{\mu}(kp^j+kl)=\frac{1}{|F_n|}[\sum_{j\in F_n\cap A}+\sum_{j\in F_n\setminus A}]\hat{\mu}(kp^j+kl)   \\
=&\frac{|F_n\cap A|}{|F_n|}\hat{\mu}(k)+\frac{1}{|F_n|}\sum_{j\in F_n\setminus A}]\hat{\mu}(kp^j+kl)\to \hat{\mu}(k)
\end{align*}
as $n\to\infty$. Hence $\hat{\mu}(k)=\hat{\mu}(k)\hat{\mu}(kl)$. This implies that $\hat{\mu}(kl)=1$ which contradict that $\mu$ is non-atomic according to
Proposition~\ref{Fouriercoefficients}.

\item Suppose $\mu$ is a weakly mixing $\times p$-invariant measure  and  there exist a nonzero integer $l$ and a F\o lner sequence $\Sigma=\{F_n\}_{n=1}^\infty$ such that
$\mu$ is $\times (p^j+l)$-invariant for all $j$ in some $A\subseteq\mathbb{N}$ with $\overline{D}_\Sigma(A)>0$.

If $\mu$ is not Lebesgue measure, then there exists nonzero $k$ in $\mathbb{Z}$ such that $\hat{\mu}(k)$ is nonzero.

By Theorem~\ref{thm: F}(2), one has
$$\lim_{n\to\infty}\frac{1}{|F_n|}\sum_{j\in F_n}|\hat{\mu}(kp^j+kl)-\hat{\mu}(k)\hat{\mu}(kl)|^2=0.$$
It follows that
\begin{align*}
0=&\limsup_{n\to\infty}\frac{1}{|F_n|}\sum_{j\in F_n}|\hat{\mu}(kp^j+kl)-\hat{\mu}(k)\hat{\mu}(kl)|^2  \\
\geq &\limsup_{n\to\infty}\frac{1}{|F_n|}\sum_{j\in F_n\cap A}|\hat{\mu}(kp^j+kl)-\hat{\mu}(k)\hat{\mu}(kl)|^2 \\
=&\limsup_{n\to\infty}\frac{1}{|F_n|}\sum_{j\in F_n\cap A}|\hat{\mu}(k)-\hat{\mu}(k)\hat{\mu}(kl)|^2  \\
=&|\hat{\mu}(k)-\hat{\mu}(k)\hat{\mu}(kl)|^2 \overline{D}_\Sigma(A).
\end{align*}
Hence $\hat{\mu}(k)-\hat{\mu}(k)\hat{\mu}(kl)=0$ which implies that $\hat{\mu}(kl)=1$. This again leads to a contradiction.

\item Assume that $\mu$ is a strongly mixing $\times p$-invariant measure and there exist a nonzero integer $l$ and an infinite  $A\subseteq\mathbb{N}$ such that
$\mu$ is $\times (p^j+l)$-invariant for all $j$ in $A$.

If $\mu$ is not Lebesgue measure, then there exists nonzero $k$ in $\mathbb{Z}$ such that $\hat{\mu}(k)$ is nonzero.

By Theorem~\ref{thm: F}(3), we have
$$\lim_{\substack{j\to\infty \\ j\in A}} \hat{\mu}(kp^j+kl)=\hat{\mu}(k)\hat{\mu}(kl).$$ On the other hand, for all $j\in A$, one has $\hat{\mu}(kp^j+kl)=\hat{\mu}(k)$.  So $\hat{\mu}(k)=\hat{\mu}(k)\hat{\mu}(kl)$. Again this leads to a contradiction.
\end{enumerate}
We finish the proof the first part of Theorem~\ref{thm: mr}.
\end{proof}

Before proceeding to the proof of the second part  of Theorem~\ref{thm: mr}, we need a lemma.

An {\bf atom} for a measure $\mu$ on a compact metrizable space $X$ is a point $x$ in $X$ such that $\mu\{x\}>0$.
\begin{lemma}~\label{lem: w}
Let  $T:X\to X$ be a continuous map on a compact metrizable space $X$. If a $T$-invariant measure $\mu$ has an atom $x$ with $\mu\{x\}<1$, then $\mu$ is not weakly mixing.
\end{lemma}
\begin{proof}
Suppose $\mu$ is weakly mixing and  has an atom $x$ with $\lambda=\mu\{x\}<1$.

Note that $\mu$ is weakly mixing if and only if
$$\lim_{n\to\infty}\frac{1}{n}\sum_{j=0}^{n-1}|\mu(T^{-j}A\cap B)-\mu(A)\mu(B)|^2=0$$ for all Borel subsets $A,B$ of $X$~\cite[Defn. 1.5(i)\&Thm. 1.24]{Walters1982}.

Choose $A=X\setminus\{x\}$ and $B=\{x\}$. Note that $\mu(T^{-j}A\cap B)$ can only have two possible values: 0 or $\lambda$, hence for all $j$, we have
$$|\mu(T^{-j}A\cap B)-\mu(A)\mu(B)|\geq \min\{\lambda(1-\lambda), \lambda-\lambda(1-\lambda)\}\geq c$$ for some constant $c>0$. This leads to a contradiction.
\end{proof}

Now we are ready to finish the proof of Theorem~\ref{thm: mr}.

\begin{proof}~[Proof of the second part of Theorem~\ref{thm: mr}]\

Suppose $\mu$ is a measure satisfying (2) or (3). By the first part of Theorem~\ref{thm: mr}, if $\mu$ is not a Lebesgue measure, then $\mu$ has an atom. By Lemma~\ref{lem: w}, we obtain that $\mu$ is a Dirac measure at some point $z$ in $\mathbb{T}$.  
\end{proof}

Next we prove the following.
\begin{theorem}~\label{prop: sg}
Let $\tau:\mathbb{Z}^+\to\mathbb{R}$ be an arbitrary increasing function with $\displaystyle\lim_{n\to\infty}\tau(n)=\infty$.
Then there exists a multiplicative semigroup $S_\tau$ of $\mathbb{Z}^+$ containing $p$ and satisfying:
\begin{enumerate}
\item $|S_\tau\cap[1,n]|\leq (\log_p n)^{\tau(n)}$;
\item the Lebesgue measure is the only non-atomic ergodic $\times p$-invariant measure which is $\times q$-invariant for all $q$ in $S_\tau$.
\end{enumerate}
In particular, there exists a multiplicative semigroup $S$ of $\mathbb{Z}^+$ containing $p$ and satisfying:
\begin{enumerate}
\item $\displaystyle\lim_{n\to\infty}\frac{\log |S\cap[1,n]|}{\log n}=0$;
\item the Lebesgue measure is the only non-atomic ergodic $\times p$-invariant measure which is $\times q$-invariant for all $q$ in $S$.
\end{enumerate}
\end{theorem}

\begin{proof}
Let $\{l_n\}_{n=1}^\infty$ be a sequence of positive integers such that $\lim_{n\to\infty} l_n=\infty$ and
define $f(m)=\sum_{n=1}^{m}l_n$ for every positive integer $m$.

Define $g(m)=\min\{\log_pN\,|\tau(N)\geq 1+f(m)\}$ for every positive integer $m$.

Define $F_n=[p^{g(n)}, p^{g(n)}+l_n]$ for every positive integer $n$. Then  $\Sigma=\{F_n\}_{n=1}^\infty$ is a F\o lner sequence in $\mathbb{N}$. Denote $\cup F_n$ by $A$.

Let $S_\tau$ be the multiplicative semigroup generated by $p$ and $p^j+1$ for all $j\in A$.

Since $D_\Sigma(A)=1$, by (1) of Theorem~\ref{thm: mr}, a non-atomic ergodic $\times p$-invariant measure which is $\times q$-invariant for all $q$ in $S_f$ must be the Lebesgue measure.

The remaining thing is to prove that $|S_\tau\cap[1,n]|\leq (\log_p n)^{\tau(n)}$.

Every  positive integer $n$ locates in $[p^{g(m)}, p^{g(m+1)})$ for some nonnegative integer $m$.

Consider $S_\tau\cap [1, n]$.

Firstly $|\{j\in A|p^j+1\leq n\}|\leq l_1+l_2+\cdots+l_m=f(m)$. This means that $S_f\cap[1,n]$ has at most $1+f(m)$ generators.

Note that $|\{k\,| p^k\leq n\}|\leq \log_p n$.   So for each generator, there are at most $\log_p n$ choices for its powers.

Since $n$ is in $[p^{g(m)}, p^{g(m+1)})$, we have $g(m)\leq\log_p n$. Then  $1+f(m)\leq \tau(n)$ by the definition of $g$.

Hence $$|S_\tau\cap [1, n]|\leq (\log_p n)^{1+f(m)}\leq(\log_p n)^{\tau(n)}.$$
This proves the first half of the theorem.

For the second half,  choose $\tau(n)=\log \log (n+3)$ for every $n$ in $\mathbb{Z}^+$. Then for $S=S_\tau$, we obtain that
$$\lim_{n\to\infty}\frac{\log |S\cap[1,n]|}{\log n}\leq \lim_{n\to\infty}\frac{[\log \log (n+3)](\log\log_p n)}{\log n}=0.$$
\end{proof}

\begin{remark}
Furstenberg's conjecture asks for measure rigidity of a non-lacunary semigroup generated by two positive integers $p, q$ and this semigroup has asymptotically $(\log n)^2$  elements in $[1,n]$. Sataev, Einsiedler and Fish prove measure rigidity of a semigroup containing asymptotically $n^\alpha$  elements in $[1,n]$ for some $0<\alpha<1$. Theorem~\ref{prop: sg} says that for an arbitrary increasing function $\tau(n)$ with $\displaystyle\lim_{n\to\infty}\tau(n)=\infty$,  there is a semigroup with asymptotically $(\log n)^{\tau(n)}$  elements in $[1,n]$ for which measure rigidity still holds.  One can choose $\tau$ such that  the semigroup $S_\tau$ is sparsely scattered in $\mathbb{Z}^+$.
\end{remark}


\begin{thebibliography}{999}

\small



\bibitem[Bow71]{Bowley1971}
T. Bowley. Extension of the Birkhoff and von  Neumannn ergodic theorems to semigroup actions.
{\it Ann. Inst. H. Poincar\'{e}}  Sect. B (N.S.) {\bf 7} (1971), 283--291.

\bibitem[EF10]{EisiendlerFish2010}
M. Einsiedler and A. Fish.  Rigidity of measures invariant under the action of a multiplicative semigroup of polynomial growth on $\mathbb{T}$. {\it Ergodic Theory Dynam. Systems} {\bf 30} (2010), no. 1, 151--157.

\bibitem[EF11]{EisiendlerWard2011}
M. Einsiedler and T. Ward. {\it Ergodic Theory with A View towards Number Theory.} Graduate Texts in Mathematics, {\bf 259}. Springer-Verlag, London, 2011.

\bibitem[Fur67]{Furstenberg1967}
H. Furstenberg. Disjointness in ergodic theory, minimal sets, and a problem in Diophantine approximation. {\it Math. Systems Theory} {\bf 1} (1967), 1--49.


\bibitem[Joh82]{Johnson1992}
A. S. A. Johnson. Measures on the circle invariant under multiplication by a nonlacunary subsemigroup of the integers. {\it Israel J. Math.} {\bf 77} (1992), no.1-2, 211--40.


\bibitem[Lyo88]{Lyons1988}
R. Lyons. On measures simultaneously 2- and 3-invariant. {\it Israel J. Math.} {\bf 61} (1988), no. 2, 219--224.

\bibitem[Rud90]{Rudolph1990}
D. J. Rudolph. $\times 2$ and $\times 3$ invariant measures and entropy. {\it Ergod. Th. and Dynam. Syst.} {\bf 10}, (1990), 395--406.

\bibitem[Sat75]{Sataev1975}
E. A. Sataev.   On measures invariant with respect to polynomial semigoups of circle transformations. {\it Uspehi Mat. Nauk}. {\bf 30} (1975), no.2, 203-204.

\bibitem[Wal82]{Walters1982}
P. Walters. {\it An Introduction to Ergodic Theory.} Graduate Texts in Mathematics, {\bf 79}. Springer-Verlag, New York-Berlin, 1982.


\end{thebibliography}
\end{document}